\documentclass{amsart}
\usepackage{amsmath}
\usepackage{amsthm}
\usepackage{amsfonts}
\usepackage{color}
\usepackage{url}
\usepackage[mathscr]{eucal}

\newtheorem{theorem}{Theorem}
\newtheorem{lemma}{Lemma}
\newtheorem{corollary}{Corollary}

\theoremstyle{definition}
\newtheorem*{definition}{Definition}
\newtheorem*{definitions}{Definitions}
\newtheorem*{remark}{Remark}

\newcommand{\R}{\mathbb{R}}

\DeclareMathOperator{\Equal}{Eq}
\DeclareMathOperator{\diam}{diam}

\newcommand{\G}{\mathscr{G}}

\newcommand{\N}{\mathbb{N}}
\newcommand{\Al}{\mathscr{A}}
\newcommand{\Bl}{\mathscr{B}}
\newcommand{\empw}{\varepsilon}
\newcommand{\sh}{d_{\text{Sh}}}
\newcommand{\ad}{{d}}
\DeclareMathOperator{\lang}{\mathcal{L}}
\DeclareMathOperator{\fr}{M}
\DeclareMathOperator{\Fr}{Fr}
\DeclareMathOperator{\dind}{\fr}
\newcommand{\adsymbol}{\Fr}

\newcommand{\ld}{\underline{d}}
\newcommand{\ud}{\overline{d}}
\newcommand{\F}{\mathscr{F}}
\DeclareMathOperator{\Ind}{\mathscr{I}}

\newcommand{\eps}{\varepsilon}

\begin{document}
\title[Entropy, chaos, and independence]{Two results on entropy, chaos, and independence in symbolic dynamics}
\markright{Entropy, chaos, and independence in symbolic dynamics}
\author{Fryderyk Falniowski}
\address[F. Falniowski]{Department of Mathematics, Cracow University of Economics,
Rakowicka~27, 31-510 Krak\'ow, Poland}
\email{fryderyk.falniowski@uek.krakow.pl}

\author{Marcin Kulczycki}
\address[M. Kulczycki]{Institute of Mathematics, Faculty of Mathematics and Computer Science, Jagiellonian University in Krak\'{o}w, ul. {\L}ojasiewicza 6, 30-348 Krak\'{o}w, Poland}
\email{Marcin.Kulczycki@im.uj.edu.pl}
\urladdr{http://www.im.uj.edu.pl/MarcinKulczycki} 

\author{Dominik Kwietniak}
\address[D. Kwietniak]{Institute of Mathematics, Faculty of Mathematics and Computer Science, Jagiellonian University in Krak\'{o}w, ul. {\L}ojasiewicza 6, 30-348 Krak\'{o}w, Poland}
\email{dominik.kwietniak@uj.edu.pl}
\urladdr{http://www.im.uj.edu.pl/DominikKwietniak} 

\author{Jian Li}
\address[J. Li]{Department of Mathematics,
Shantou University, Shantou, Guangdong 515063, P.R. China}
\email{lijian09@mail.ustc.edu.cn}

\begin{abstract}
We survey the connections between entropy, chaos, and independence in topological dynamics.
We present extensions of two classical results placing the following notions in the context of symbolic dynamics:
\begin{enumerate}
\item Equivalence of positive entropy and the existence of a large (in terms of asymptotic and Shnirelman densities) set of combinatorial independence for shift spaces.
\item Existence of a mixing shift space with a dense set of periodic points with topological entropy zero and without ergodic measure with full support, nor any distributionally chaotic pair.
\end{enumerate}
Our proofs are new and yield conclusions stronger than what was known before.
\end{abstract}
\keywords{topological entropy, detrministic dynamical system, independence, Devaney chaos, Li-Yorke chaos, distributional chaos, $\omega$-chaos}
\subjclass[2010]{37B40, 37B10}
\maketitle

\section{Introduction.}

Furstenberg \cite[p. 38]{Fur67} calls a dynamical system \emph{deterministic} if its \emph{topological entropy} vanishes.
One may argue that the future of a deterministic dynamical system can be predicted if its past is known (see \cite[Chapter 7]{Weiss}).
In a similar way \emph{positive entropy} may be related to \emph{randomness} and \emph{chaos}.

This article surveys the relations between entropy, independence, and chaos, adding a new twist to two classical results.

The first result describes in what sense positive entropy may be understood as independence. It turns out that for a dynamical system there is a deep connection between positive entropy and randomness defined through the notion of the \emph{combinatorial independence}. General results of this kind can be found in \cite{KerrLi}. A special instance of this result (see \cite[Theorem 8.1]{Weiss}) says that a symbolic dynamical system (a shift space) $X\subset\{0,1\}^\N$ has positive entropy if and only if there is a large set of indices along which points from $X$ behave like $\{0,1\}$-valued independent random variables. We strengthen this theorem by proving that positive topological entropy of a shift space is equivalent to the existence of an independence set for which asymptotic \textbf{and} Shnirelman densities are equal and positive. This shows that, for a shift space with positive entropy, one can find an independence set which is large and structured. One can easily adapt our reasoning to the general context.

Before we describe our second result, let us consider the following properties that a dynamical system $(X,T)$ may have.
\begin{enumerate}
\item \label{ERG} There exists an ergodic $T$-invariant Borel probability measure on $X$ with full support (that is, an invariant measure positive on all nonempty open subsets of $X$).
\item \label{TRANS} The system $(X,T)$ is \emph{topologically transitive}, that is, for any pair of nonempty open sets
$U, V \subset X$ there exists a positive integer $n$ with $T^{-n}(U)\cap V\neq \emptyset$.
\item \label{FULL} There exists a $T$-invariant (not necessarily ergodic) Borel probability measure on $X$ with full support.
\end{enumerate}
It is straightforward that \eqref{ERG} implies \eqref{TRANS} and, trivially, \eqref{FULL}. It is also relatively easy to prove that \eqref{TRANS} does not imply \eqref{FULL} and hence neither implies \eqref{ERG}. According to \cite{Weiss73}, H.~Keynes asked whether \eqref{TRANS} and \eqref{FULL} together imply \eqref{ERG}. B.~Weiss answered this question negatively in \cite{Weiss73}. He constructed a transitive shift space with a dense set of periodic points and no ergodic invariant measure of full support. He also noted, omitting the proof, that his example is in fact topologically mixing and has zero topological entropy. Furthermore, Weiss conjectured \cite[Remark 1]{Weiss73} that the only ergodic measures for his system are those concentrated on the orbit of a single periodic point. Note that Huang and Ye constructed a uniformly positive entropy system without an ergodic invariant measure of full support (see \cite[Theorem 9.6]{HY06}).

We provide an alternative construction of a similar example. Our method allows us to prove  that the system we have defined possesses all the properties proved or conjectured about the example from \cite{Weiss73}. Furthermore, we show that our example does not possess any distributionally chaotic pairs. This extends a result by Oprocha \cite{Oprocha06} who sketched a construction of a Devaney chaotic system without any DC$1$ distributionally chaotic pairs.
In addition, we show that our example is $\omega$-chaotic, but not $\omega^*$-chaotic. To do this we provide a new method of constructing $\omega$-chaotic sets.

\section{Notation and definitions.}

Let $\N$ denote the set of \emph{positive} integers.
We denote the number of elements of a finite set $A$ by $|A|$. 
Given any real number $x$, we write  $\lfloor x\rfloor$ for the largest integer not greater than $x$. A sequence of real numbers $\{a_n\}_{n=1}^\infty$ is \emph{subadditive} if $a_{m+n}\leq a_m+a_n$ for all $m,n\in\N$.
We recall that Fekete's Lemma states that if a sequence $\{a_n\}_{n=1}^\infty$ of nonnegative numbers is subadditive, then the sequence $\{\frac{1}{n}a_n\}_{n=1}^\infty$ converges to a limit equal to the infimum of the terms of this sequence.

A \emph{(topological) dynamical system} is a pair $(X,T)$,
where $X$ is a compact metric space and $T\colon X\to X$ is a continuous map.
By $\rho$ we denote a compatible metric for $X$. A dynamical system $(X,T)$ is \emph{transitive} if for any nonempty open sets $U,V\subset X$ there is $n\in\N$
such that $T^{-n}(V)\cap U\neq\emptyset$. We say that $(X,T)$ is \emph{topologically mixing} if for any nonempty open sets $U,V\subset X$ there is $N\in\N$
such that $T^{-n}(V)\cap U\neq\emptyset$ for all $n\ge N$.
The set of limit points of the sequence $\{T^n(x)\}_{n\in\mathbb{N}}$ is called the \emph{$\omega$-limit set} of $x$ and denoted by $\omega_T(x)$. A set $M\subset X$ is called \emph{minimal} if it is nonempty, closed, invariant, and contains no proper subset with these properties.

We assume the reader is familiar with the basic notions of topological dynamics and ergodic theory (see \cite{Walters}). In particular, we assume that the reader knows the definition of topological entropy.

\section{Definitions of chaos and their relations with topological entropy}
In this section we collect some definitions of chaos and survey connections between them and topological entropy.

\subsection{Li-Yorke chaos}
Let $(X,T)$ be a dynamical system and $(x,y)\in X\times X$.
We say that $(x,y)$ is a \emph{Li-Yorke pair}\footnote{Such pairs appeared in the seminal paper of Li and Yorke \cite{LY}.} if
\begin{align*}
\liminf_{n\to\infty}\rho(T^n(x),T^n(y))&=0,\\
\limsup_{n\to\infty}\rho(T^n(x),T^n(y))&>0.
\end{align*}
A dynamical system $(X,T)$ is \emph{Li-Yorke chaotic} if there is an uncountable set $S\subset X$ such that every pair $(x,y)$ with $x,y\in S$ and $x\neq y$ is a Li-Yorke pair. A nice survey of properties of Li-Yorke chaotic systems can be found in \cite{BHS}.

\subsection{Distributional chaos}
Given $x,y\in X$ and $t\in\R$ we define an \emph{upper} and \emph{lower distribution} function
by
\begin{align*}
  F_{xy}(t) &= \liminf_{n\to\infty}\frac{1}{n}\left| \left\{0\le j \le n-1 : \rho(T^j(x),T^j(y))<t\right\}\right|,\\
F^*_{xy}(t) &= \limsup_{n\to\infty}\frac{1}{n}\left| \left\{0\le j \le n-1 : \rho(T^j(x),T^j(y))<t\right\}\right|.
\end{align*}
Clearly, $t\mapsto F_{xy}(t)$ and $t\mapsto F^*_{xy}(t)$ are nondecreasing, $0\le F_{xy}(t)\le F^*_{xy}(t)\le 1$ for all $t\in\R$,  $F_{xy}(t)= F^*_{xy}(t)=0$ for all $t\le 0$, and $F_{xy}(t)= F^*_{xy}(t)=1$ for all $t>\diam X$.
Distributional chaos was introduced in \cite{SS}.
Following \cite{BSS} we say that a pair $(x,y)\in X\times X$  is a DC$1$-\emph{pair} if
$F^*_{xy}(t)=1$ for all $t>0$, and $F_{xy}(s)=0$ for some $s>0$.
A pair $(x,y)$ is a DC$2$-\emph{pair} if
$F^*_{xy}(t)=1$ for all $t>0$, and $F_{xy}(s)<1$ for some $s>0$.
Finally, a DC$3$-\emph{pair} is a pair $(x,y)\in X\times X$ such that
$F_{xy}(t) < F^*_{xy}(t)$ for all $t$ in some interval of positive length.

Let $i\in\{1,2,3\}$. The dynamical system $(X,T)$ is \emph{distributionally chaotic of type $i$}
(or DC\emph{$i$-chaotic} for short), if there is an uncountable
set $S\subset X$ such that any pair of distinct points from $S$ is a DC\emph{i}-pair.

\subsection{Devaney chaos} Devaney calls a dynamical system $(X,T)$  \emph{chaotic}
if $T$ is transitive, $T$-periodic points are dense in $X$, and $T$ has sensitive dependence on
initial conditions (see \cite{Devaney}). It turns out that if $X$ is infinite, then the sensitive dependence
follows from the other two conditions (see \cite{BBCDS,GW93}). Since we restrict our attention to
compact metric spaces without isolated points we say that $(X,T)$ is
Devaney chaotic if $T$ is transitive and $T$-periodic points are dense in $X$. Such systems are also known as
\emph{$P$-systems} (see \cite{GW93}).
\subsection{$\omega$-chaos}
We say that a dynamical system $(X,T)$ is \emph{$\omega$-chaotic} if there exists an uncountable set $S\subset X$ such that for any $x,y\in S$ with $x\neq y$ we have
\begin{itemize}
\item $\omega_T(x)\backslash \omega_T(y)$ is uncountable,
\item $\omega_T(x)\cap\omega_T(y)$ is nonempty,
\item $\omega_T(x)$ is not contained in the set of periodic points.
\end{itemize}
Li \cite{Li} introduced $\omega$-chaos and proved that it is equivalent to positive topological entropy for interval maps. In \cite{LiETDS} Li defined a variant of $\omega$-chaos, called \emph{$\omega^*$-chaos}.
A dynamical system $(X,T)$ is \emph{$\omega^*$-chaotic} if there exists an uncountable set $S\subset X$ such that for any $x,y\in S$ with $x\neq y$ we have
\begin{itemize}
\item $\omega_T(x)\backslash \omega_T(y)$ contains an infinite minimal set,
\item $\omega_T(x)\cap\omega_T(y)$ is nonempty.
\end{itemize}
Since an infinite minimal set is uncountable and does not contain a periodic point, $\omega^*$ implies $\omega$-chaos. All examples of $\omega$-chaotic maps known to us are in fact $\omega^*$-chaotic. As is common for various notions of chaos in a general setting of an arbitrary compact metric space, $\omega$-chaos is independent of most other notions of chaos.

\subsection{Entropy vs. chaos}

There is no connection between positive topological entropy and Devaney chaos. It is known that Devaney chaos implies Li-Yorke chaos \cite{HY02} and does not imply distributional chaos \cite{Oprocha06}. The example in \cite{Weiss73} shows that Devaney chaos (hence Li-Yorke chaos as well by \cite{HY02}) does not imply positive entropy. In general, positive entropy neither implies topological transitivity nor existence of periodic points.  Nevertheless,
Li \cite{Li} proved that for maps on the interval positive topological entropy is equivalent to the existence of a subsystem chaotic in the sense of Devaney. The question whether positive topological entropy implies
Li-Yorke chaos or DC$2$ distributional chaos remained open for some time, but eventually both implications turned out to be true; see \cite{BGKM} and \cite{Downarowicz}, respectively. The consequences of positive topological entropy for dynamics of pairs and tuples were also examined in \cite{BH,HLY}.
Distributional chaos of type $2$ implies Li-Yorke chaos by definition. The converse implication is not true, because if $X=[0,1]$, then by \cite{SS} distributional chaos DC$3$ is equivalent to the positive topological entropy, and there are Li-Yorke chaotic interval maps with zero topological entropy (as was shown independently by Sm\'{\i}tal \cite{Smital} and Xiong \cite{Xiong}). Piku{\l}a proved that there is no connection between positive topological entropy and DC$1$ distributional chaos \cite{Pikula}, therefore Li-Yorke chaos doesn't imply positive topological entropy, either. Piku{\l}a \cite{Pikula} also constructed an example of $\omega$-chaotic dynamical system without Li-Yorke pairs. Since there exist minimal systems with positive topological entropy, none of the following properties: positive topological entropy, distributional chaos, Li-Yorke chaos imply $\omega$-chaos (just note that if $(X,T)$ is minimal, then $\omega_T(x)=X$ for every $x\in X$). Downarowicz and Ye \cite{DY} constructed a transitive dynamical system $(X,T)$ which is Devaney chaotic and every point $x\in X$ is either transitive ($\omega_T(x)=X$) or periodic. Such a system cannot be $\omega$-chaotic.

All above is only a glimpse of the vast literature of the subject. For deeper discussion of these matters we refer the reader to the excellent surveys
by Blanchard \cite{Blanchard}, Glasner and Ye \cite{GY}, Li and Ye \cite{LY14}.

\section{Symbolic dynamics}

With regard to symbolic dynamics we follow the notation and terminology of Lind and Marcus~\cite{LM95} as closely as possible, except that we consider only one-sided shifts.

\begin{definitions}[Full shifts]
Let $\Al$ be a nonempty finite set. We call $\Al$ the \emph{alphabet} and elements of $\Al$ are \emph{symbols}.
The \emph{full $\Al$-shift} is the set \[\Al^\N=\{x=\{x_i\}_{i=1}^\infty : x_i\in\Al\text{ for all } i\in\N\}.\]
We equip $\Al$ with the discrete topology and $\Al^\N$ with the product topology. We usually write an element of $\Al^\N$ as $x=\{x_i\}_{i=1}^\infty=x_1x_2x_3\ldots$. Then $\Al^\N$ is a compact topological space
and the formula
\begin{equation}\label{eq:rho}
\rho(x,y)=\begin{cases}
0,&\text{if }x=y,\\
2^{-\min\{j\in\N:x_j\neq y_j\}}, &\text{if }x\neq y,
\end{cases}
\end{equation}
defines a metric on $\Al^\N$ which is compatible with the product topology.
The \emph{shift map} $\sigma\colon\Al^\N\rightarrow\Al^\N$ is a~continuous transformation given by
\[
x=\{x_i\}_{i=1}^\infty\mapsto \sigma(x)=\{x_{i+1}\}_{i=1}^\infty.
\]
That is, $\sigma(x)$ is the sequence obtained by dropping the first symbol of $x$. A \emph{full $r$-shift} is the full shift over the alphabet $\{0,1,\ldots,r-1\}$ and a \emph{full binary shift} is the full $2$-shift.
\end{definitions}
\begin{definitions}[Blocks]
A \emph{block over $\Al$} is a finite sequence of symbols and its \emph{length} is the number of its symbols. An \emph{$n$-block} stands for a block of length $n$. The \emph{empty block}, denoted by $\empw$, is the unique block with no symbols and length zero. The set of all blocks over $\Al$ (including $\empw$) is denoted by $\Al^*$. The \emph{concatenation} of two blocks $u=a_1\ldots a_k$ and $v=b_1\ldots b_l$ is the block $uv=a_1\ldots a_kb_1\ldots b_l$. We write $u^n$ for the concatenation of $n\ge 1$ copies of a block $u$ and $u^\infty$ for the sequence $uuu\ldots\in\Al^\N$.

By $x_{[i,j]}$ we denote the block $x_ix_{i+1}\ldots x_j$, where $1\le i \le j$ and $x=(x_i)_{i=1}^\infty\in\Al^\N$. We say that a block $w\in A^*$ \emph{occurs in $x$} and $x$ \emph{contains} $w$ if $w=x_{[i,j]}$ for some integers $1\le i \le j$. Note that $\empw$ occurs in every sequence from $\Al^\N$. Similarly, given an $n$-block $w=w_1\ldots w_n \in\Al^*$ we define $w_{[i,j]}=w_iw_{i+1}\ldots w_j\in\Al^*$ for each $1\le i\le j\le n$.
A \emph{prefix} of a block $z\in\Al^*$ is any $u$ such that $z=uv$ for some $v\in\Al^*$.
\end{definitions}

\begin{definitions}[Shift spaces]
Given any collection $\F$ of blocks over $\Al$ (i.e., a~subset of $\Al^*$) we define a \emph{shift space specified by $\F$}, denoted by $X_\F$, as the set of all sequences from $\Al^\N$ which do not contain any blocks from $\F$. We say that $\F$ is a collection of \emph{forbidden blocks for $X_\F$}.

A~\emph{shift space} is a set $X\subset\Al^\N$ such that $X=X_\F$ for some $\F\subset\Al^*$. A~\emph{binary shift space} is a shift space over the alphabet $\{0,1\}$. Equivalently, $X\subset\Al^\N$ is a shift space if it is a closed and $\sigma$-invariant subset of $\subset\Al^\N$.
\end{definitions}

\begin{definition}[Language of a shift space]
The \emph{language} of a shift space $X$ over $\Al$ is the set of blocks over $\Al$ which do occur in some sequence $x\in X$. We denote it by $\Bl(X)$. We write $\Bl_n(X)$ for the set of all $n$-blocks contained in $\Bl(X)$. Similarly, $\Bl(x)\subset \Al^*$ is the collection of all blocks occurring in $x\in\Al^\N$. The language of a shift space determines the shift space: $x\in\Al^\N$ belongs to a shift space $X$ if and only if for every $k\in\N$ the initial block $x_{[1,k]}$ is in $\Bl(X)$.
\end{definition}

A \emph{cylinder} of an $n$-word $w\in \Al^*$ in a shift space $X$ is the set
\[
[w]_X=\{x\in X: x_{[1,n]}=w\}.
\]
If a space $X$ is clear from the context, we call $[w]_X$ the cylinder of $w$. The collection of all cylinders
$\{[w]_X:w\in\Bl(X)\}$ is a basis of the topology of $X$.
If $\mathcal{L}$ is a language of some shift space over $\Al$, then $\mathcal{L}$ is
\emph{factorial}, meaning that if $u\in\mathcal{L}$ and $u=vw$ for some blocks $v,w\in \Al^*$, then both $v$ and $w$ also belong to $\mathcal{L}$, and
\emph{prolongable}, meaning that for every block $u$ in $\mathcal{L}$ there is a~symbol $a\in \Al$ such that $ua$ also belongs to $\mathcal{L}$.

Conversely, every factorial and prolongable subset $\mathcal{L}\subset\Al^*$ determines a shift space $X_\mathcal{L}$ such that $\mathcal{L}$ is the language of $X$ (see \cite[Proposition 1.3.4]{LM95}). A collection of forbidden blocks defining $X$ is $\F=\Al^*\setminus\mathcal{L}$. A point $x\in \Al^\N$ is in $X_\mathcal{L}$ if and only if $x_{[i,j]}\in\mathcal{L}$ for all $i,j\in\N$ with $ i < j$.

We will use superscripts in brackets to denote indices for sequences of blocks in $\Al^*$ or points in $\Al^\N$.
That is, we write $\{w^{(n)}\}_{n=1}^\infty$ for a sequence of blocks and we use subscripts for enumerating symbols in the block: $w^{(n)}=w^{(n)}_1w^{(n)}_2\ldots w^{(n)}_k$.


In its full generality, the concept of entropy was defined by Adler, Konheim and McAndrew \cite{AKM} for
any continuous map $T\colon X\to X$ on an arbitrary compact topological space $X$. The definition below, specific to symbolic dynamics, is equivalent to the general one.

\begin{definition}[Entropy]
Let $X\subset\Al^\N$ be a nonempty shift space and let $m,n\in\N$. Observe that $|\Bl_{m+n}(X)|\le|\Bl_{m}(X)|\cdot |\Bl_{n}(X)|$, and hence
\[
\log |\Bl_{m+n}(X)|\le \log|\Bl_{m}(X)|+\log |\Bl_{n}(X)|.
\]
Using Fekete's lemma we define the \emph{entropy of $X$}, denoted by $h(X)$, as
\[
h(X)=\lim_{n\to\infty} \frac{1}{n}\log |\Bl_n(X)|=\inf_{n\ge 1} \frac{1}{n}\log |\Bl_n(X) |.
\]
\end{definition}

\section{Entropy and independence in symbolic dynamics}
Independence is certainly among the most popular terms used to single out an interesting mathematical phenomenon.
Consider the full shift over $\{0,1\}$. Clearly, the full shift is as random as possible --- it is a model containing all possible outcomes of an infinite sequence of fair coin tosses (we set $x_i=1$ if the coin turns heads on the $i$-th toss, and $x_i=0$ otherwise). Anything that can happen is encoded in some point $x$ --- more precisely, for any $J\subset\N$ and any assignment $\varphi\colon J\to \Al$ of $0$'s and $1$'s to elements of $J$ there is a point $x$ in the full shift realizing this assignment, that is, $x_i=\varphi(i)$ for all $i\in J$. For a general shift space $X$ over $\{0,1\}$ the above procedure may not work for every set $J\subset\N$.
But one may still consider the following problem: let $X$ be a shift space over $\Al$. Assume that someone picks a set $J\subset\N$ and for each $i\in J$ chooses a symbol $\varphi(i)\in\Al$ (the choice may be random). Does there exist a point $x\in X$ whose $i$-th coordinate is $\varphi(i)$ for every $i\in J$? If the answer is positive for every assignment $\varphi$, then we say that $X$ is \emph{independent over $J$}. Now we may ``measure'' the randomness of $X$ by the size of the ``largest'' independent set $J\subset\N$.

\begin{definition}[Independence set for a shift space]
We say that a set $J\subset\N$ is an \emph{independence set} for a shift space $X\subset\Al^\N$ if for every function $\varphi\colon J\to \Al$ there is
a point $x=\{x_j\}_{j=1}^\infty\in X$ such that $x_i=\varphi(i)$ for every $i\in J$.
\end{definition}

Note that if $J\subset\N$ is an independence set for a shift space $X$, then so is every subset of $J$.




\subsection{Asymptotic density and Shnirelman density}

The most natural way to describe the size of an infinite subset of $\N$ is the \emph{asymptotic density}.
It is translation invariant and it is invariant under the exclusion or inclusion of
finitely many elements. A similar notion is the \emph{Shnirelman density}, which formally is also a limit,  but
it gives information about the structure of $A\cap\{1,\ldots,n\}$ for every $n\in\N$.
Although it seems to be less natural a~concept than asymptotic density, Shnirelman's notion proved invaluable to his approach to the Goldbach problem.

\begin{definitions}[Densities]
A set $A\subset\N$ has \emph{asymptotic density $\alpha$} if the limit
\[
d(A)=\lim_{n\to\infty}\frac{|A\cap\{1,2,\ldots,n\}|}{n}
\]
exists and is equal to $\alpha$.

The \emph{Shnirelman density} $\sh(A)$ of a set $A\subset \N$ is defined by
\[
\sh(A)=\inf \bigg\{\frac{|A\cap\{1,2,\ldots,n\}|}{n}:n\in\N\bigg\}.
\]
\end{definitions}

\subsection{Positive entropy is equivalent to independence}
Both notions measure how dense a set is, and a set with any of these densities positive may be called large.
We will prove that positive topological entropy is equivalent to existence of an independence set for which asymptotic \textbf{and} Shnirelman densities are equal and positive. This shows that for a shift space with positive entropy one can find an independence set which is large and structured
(there exists $0<\alpha\le 1$ such that our set occupies at least $\alpha$ proportion of $\{1,\ldots,n\}$ for \emph{every} $n\in\N$). Note that positive asymptotic density of a set $A$ tells us only that $|A\cap \{1,\ldots,n\}|$ behaves like $\alpha n$ for some $\alpha>0$ and $n$ \emph{large enough}.

\begin{theorem}\label{main}
Let $X$ be a binary shift. Then the entropy of $X$ is positive if and only if $X$ is independent over a set whose asymptotic density exists, is positive, and is equal to its Shnirelman density.
\end{theorem}

This is a strengthening of \cite[Theorem 8.1]{Weiss}. A general result applicable to all dynamical systems was proved by Glasner and Weiss \cite{GW} and Huang and Ye \cite[Theorem 7.3]{HY06}.
Kerr and Li extended it further in \cite{KerrLi}. We add the Shnirelman density to the picture, which shows that the independence set is even more structured. A~similar result holds for more numerous alphabets (cf. \cite[Theorem 8.3]{Weiss}).
We state it without a proof.

\begin{theorem}\label{main2}
Let $r\ge 2$ and $X$ be a shift space over $\Al=\{0,1,\ldots,r-1\}$. Then $h(X)>\log(r-1)$ if and only if $X$ is independent over a set $A$ whose asymptotic density exists, is positive, and is equal to its Shnirelman density.
\end{theorem}

Before the proof we introduce all the necessary tools.

\subsection{Limiting frequency}
Our main technical tool is the limiting frequency.
\begin{definition}[Limiting frequency]
Let $\Al$ be a finite alphabet and $X$ be a shift space over $\Al$. For every symbol $a\in\Al$
and every point $x\in X$ we define the \emph{characteristic set $\chi_a(x)$ of $a$ in $x$}
as the set of positions at which $a$ appears in $x$, that is,
\[
\chi_a(x)=\{j\in\N:x_j=a\}.
\]
Let $w=w_1\ldots w_k$ and let $||w||_a$ denote the number of $a$'s in $w$, that is
\[
||w||_a=|\{1\le j \le k:w_j=a\}|.
\]
Let $\dind^a_k(X)$ be the maximal number of occurrences of the symbol $a$ among all blocks $w\in\Bl_k(X)$, that is,
\[
\dind^a_k(X)= \max\left\{||w||_a\,:\,w\in\Bl_k(X)\right\}.
\]
The sequence $\{\dind^a_k(X)\}_{k=1}^\infty$ is non-negative and subadditive, that is
$$0\le \dind^a_{m+n}(X)\le \dind^a_m(X)+\dind^a_n(X)$$
for any positive integers $m$ and $n$. By Fekete's Lemma the sequence $\{\dind^a_k(X)/k\}_{k=1}^\infty$ converges to its greatest lower bound. We call this limit the \emph{limiting frequency of $a$ in $X$} and denote it by
\begin{equation}\label{eq:adsymbol}
\adsymbol_a(X)=\lim_{k\to\infty}\frac{\dind^a_k(X)}{k}=\inf_{k\ge 1}\frac{\dind^a_k(X)}{k}.
\end{equation}
\end{definition}

It can be shown that a limiting frequency of a symbol $a$
is precisely the maximum of measures of the cylinder of $a$ with respect to ergodic invariant measures supported on $X$.
The next theorem is a strengthening of a well-known result about the ordinary density. It also follows\footnote{We are indebeted to Vitaly Bergelson for pointing this out.} from the Ruzsa Theorem \cite[Theorem 4, p. 323]{Ruzsa} (see also \cite{Peres}).
\begin{theorem}\label{thm:density}
Let $X$ be a shift space over an alphabet $\Al$. Then for every symbol $a\in\Al$
there exists a point $\omega_a\in X$ such that
\[
\sh(\chi_a(\omega_a))=\ad(\chi_a(\omega_a))=\adsymbol_a(X).
\]
\end{theorem}
\begin{proof}
Clearly, $\adsymbol_a(X) =0$ implies $\sh(\chi_a(x))=\ad(\chi_a(x))=0$ for all $x\in X$. We may therefore assume that $\adsymbol_a(X)>0$.

For every $m>0$ let $\bar{w}^{(m)}=\bar{w}^{(m)}_1\ldots \bar{w}^{(m)}_m\in\Bl_m(X)$ be a block which attains the maximal number of occurrences of the symbol $a$, that is,
\[
||\bar{w}^{(m)}||_a=\dind^a_m(X)= \max\left\{||w||_a:w\in\Bl_m(X)\right\}.
\]

We claim that for each integer $k>0$ there exists a block $w^{(k)}\in\Bl_k(X)$ such that for all $1\le j \le k$ we have
\begin{equation}\label{eq:star}
j\cdot \adsymbol_a(X)\le ||w^{(k)}_1\ldots w^{(k)}_j||_a.
\end{equation}
In other words, we claim that given $k>0$ we can find a word of length $k$ such that for each $1\le j \le k$ the average number of occurrences of $a$ in the prefix
$w^{(k)}_1\ldots w^{(k)}_j$ of $w^{(k)}$ is not smaller than the limiting frequency of $a$.

For the proof of the claim, assume on the contrary that \eqref{eq:star} does not hold for some $k>0$. This means that every block $w$ of length greater than or equal to $k$ has a prefix of length $1\le j\le k$ for which \eqref{eq:star} fails, that is, such that
\[
\frac{1}{j} ||w_1\ldots w_j||_a < \adsymbol_a(X)\le 1.
\]
Note that the left hand side of the above inequality is always a fraction of the form $b/c$, where $0\le b < c\le k$. Since $k$ is fixed, there are only finitely many such fractions, hence the number $\alpha_0=\min\{\adsymbol_a(X)- b/c>0:0\le b<c\le k \}$ is well-defined (the minimum exists) and is positive.

Take a positive integer $N$ such that $\lfloor N\alpha_0\rfloor=1$. Since we assumed that our claim fails, and since $p=k^2N+1$, we can chop the block $\bar{w}^{(p)}$ defined above into at least $Nk+1$ pieces, each of length $k$ at most, and for all but at most one of them \eqref{eq:star} fails. To see it, we can imagine that we apply to the block $\bar{w}^{(p)}$ the following procedure: for the first step, we set $l(0)=0$. Then we find the smallest $j$ such that \eqref{eq:star} fails for $\bar{w}^{(p)}$. By our assumption there is such $j$, moreover, $j\le k$. We set $l(1)=j$ and define
\[
\alpha_1= \adsymbol_a(X)-\frac{1}{j} ||\bar{w}^{(p)}_{1}\bar{w}^{(p)}_{2}\ldots \bar{w}^{(p)}_j||_a.
\]
Clearly $\alpha_0\le \alpha_1$. For the next step we consider
\[
v=\bar{w}^{(p)}_{j+1}\bar{w}^{(p)}_{j+2}\ldots \bar{w}^{(p)}_p,
\]
that is, we chop off the prefix of length $j$ from $\bar{w}^{(p)}$. We proceed inductively. Assume that we have performed $s$ steps. This means that $l(0)<\ldots<l(s)$ and $\alpha(1),\ldots,\alpha(s)$ are defined for some $s\ge 1$. If $p-l(s)<k$, then we set $l(s+1)=p$ and we finish the construction. Otherwise, we have $p-l(s)\ge k$ and we consider \[v=\bar{w}^{(p)}_{l(s)+1}\bar{w}^{(p)}_{l(s)+2}\ldots \bar{w}^{(p)}_p.\] Then we find the smallest $j$ such that \eqref{eq:star} fails for $v$ . We set $l(s+1)=l(s)+j$ and
\[
\alpha_{s+1}=  \adsymbol_a(X)-\frac{1}{l(s+1)-l(s)}||\bar{w}^{(p)}_{[l(s)+1,l(s+1)]}||_a\ge\alpha_0>0.
\]
Let $t$ denote the number of steps we have performed. We have found a~strictly increasing sequence of integers $\{l(s)\}_{s=0}^{t+1}$ as well as a sequence of positive real numbers $\{\alpha_s\}_{s=1}^t$ such that $l(0)=0$, $l(t+1)=p$, for every $j=1,2\ldots,t+1$ we have $0\le l({j})-l(j-1)\le k$, and
\begin{multline}\label{eq:x}
||\bar{w}^{(p)}_{[l(j-1)+1,l(j)]}||_a  = (l({j})-l(j-1))(\adsymbol_a(X)-\alpha_{j})\le\\ (l({j})-l(j-1))(\adsymbol_a(X) - \alpha_0) \text{ for }j=1,\ldots,t.
\end{multline}
Note that it may happen that $l(t)=l(t+1)=p$, but the inequality $||\bar{w}^{(p)}_{[l(t)+1,p]}||_a<k$ is valid regardless (if $l(t)+1>p$, then we agree to identify $\bar{w}^{(p)}_{[l(t)+1,p]}$ with the empty block).
Clearly, we have $t\ge kN$, hence by the inequality in \eqref{eq:x} and the definition of $N$ we have
\begin{align*}
||\bar{w}^{(p)}||_a&=\bigg(\sum_{j=1}^{t} ||\bar{w}^{(p)}_{[l(j-1)+1,l(j)]}||_a \bigg) + ||\bar{w}^{(p)}_{[l(t)+1,p]}||_a\\
                           &\le p(\adsymbol_a(X)-\alpha_0)+k\le p\adsymbol_a(X)-k^2N\alpha_0+k<p\adsymbol_a(X).
\end{align*}
Now we have the required contradiction, since $||\bar{w}^{(p)}||_a$ should be greater or equal to $p\adsymbol_a(X)$.

We proved our claim is true, hence we have a sequence of blocks $\{w^{(k)}\}_{k=1}^\infty$, with $w^{(k)}$ from $\Bl_k(X)$, satisfying \eqref{eq:star}.
We will use\footnote{This can be obtained easily by using a standard compactness argument, but the referee has encouraged us to present a completely elementary proof of this fact.} the $w^{(k)}$ to find a point $x\in X$ such that for each $k>0$ the block $x_{[1,k]}$ is a prefix of some $w^{(l(k))}$ with $l(k)\ge k$.

For $k\in\N$ we inductively find a decreasing sequence $S_k$ of subsets of $\N$ so that for any $n\in S_k$ with $n\ge k$ all prefixes $w^{(n)}_{[1,k]}$ are equal. We begin by picking a symbol $a\in\Al$ so that $w^{(k)}_1=a$ for infinitely many $k$'s. Let $S_1=\{k\in\N:w^{(k)}_1=a\}$. Next, since the set of possible blocks of length $2$ in $X$ is finite, there is an infinite subset $S_2\subset S_1$ such that the prefix $w^{(n)}_{[1,2]}$ is the same for all $n\in S_2$. Continuing this way, we find for each $k\ge 2$ and infinite set $S_k\subset S_{k-1}$ so that all blocks $w^{(n)}_{[1,k]}$ are equal for $n\in S_k$. Define $x$ to be the sequence with $x_k=w^{(n)}_k$ for some $n\in S_k$ with $n\ge k$. Note that then $x_{[1,k]}=w^{(n)}_[1,k]$ for every $n\in S_k\setminus\{1,\ldots,k-1\}$ since $S_k\subset S_{k-1}$ for $k\ge 2$. Therefore for each $k>0$ the block $x_{[1,k]}$ is a prefix of some $w^{(l(k))}$ with $l(k)\ge k$. In particular, $x_{[1,k]}\in\Bl(X)$ and hence $x\in X$.

For every $k\in\N$ we have
\begin{equation}\label{ineq}
\adsymbol_a(X) \le \frac{1}{k}||x_{[1,k]}||_a\le \frac{1}{k}\dind^a_k(X),
\end{equation}
where the first inequality above is a consequence of the claim and the construction of the point $x$, while the second inequality follows from the definition of $\dind^a_k(X)$. Taking the infima over all $k\in\N$ of all terms in \eqref{ineq} and applying \eqref{eq:adsymbol} we obtain
\[
\adsymbol_a(X) \le \inf_{k\in\N}\frac{1}{k}||x_{[1,k]}||_a=\sh(\{j\in\N: x_j=a\})\le\inf_{k\in\N}\frac{1}{k}\dind_a^k(X)=\adsymbol_a(X).
\]
Since the values $\dind_a^k(X)/k$ converge to $\adsymbol_a(X)$ as $k$ goes to infinity (see \eqref{eq:adsymbol}) we conclude that
\[
\lim_{k\to\infty}\frac{1}{k}||x_{[1,k]}||_a= \ad(\{j\in\N: x_j=a\})=\sh(\{j\in\N: x_j=a\})=\adsymbol_a(X).\qedhere
\]
\end{proof}

\subsection{Some combinatorial Lemmas}
The following lemma is well-known.
\begin{lemma}\label{suma}
Let $0<\eps\le 1/2$ and $n\ge 1$. Then
\[
\sum_{j=0}^{\lfloor n \eps \rfloor}{\binom{n}{j}}\le 2^{n\cdot H(\eps)},
\]
where $H(\eps)=-\eps\log \eps - (1-\eps)\log (1-\eps)$.
\end{lemma}

\begin{definition}[Independence sets for blocks]
Let $\F$ be a (possibly empty) family of binary blocks of length $n\ge 0$. We say that $\F$ is \emph{independent over a set $J\subset\N$} and $J$ is an \emph{independence set for $\F$}
if for each map $\varphi\colon J\to\{0,1\}$ there is a block $w\in\F$ whose $i$-th symbol is $\varphi(i)$ for every $i\in J$.

We denote the collection of all sets of independence for $\F$ by $\Ind(\F)$. We assume the convention that the empty set is a set of independence for every (including empty) family of $n$-blocks.
\end{definition}

We will need a variant of the famous Sauer-Perles-Shelah Lemma (Lemma \ref{indep} below, see \cite{Sauer,Shelah}).
The proof follows \cite{shattering}, where the next lemma (stated for sets) is attributed to Alain Pajor \cite{Pajor}.

\begin{lemma}
Let $\F$ be a family of binary blocks of length $n\ge 0$. Then $|\Ind(\F)|\ge|\F|$.
\end{lemma}
\begin{proof}
The proof goes by induction on $n$. For $n=0$ the situation is clear: $\F\subset\{\empw\}$, hence $|\F|\le 1$ and $\Ind(\F)\ge 1$ because of the empty set.

Assume that the lemma holds for all families of $n$-blocks for some $n\ge 0$. Consider a family $\F\subset\{0,1\}^{n+1}$.
Let $\F_0$ be the family of prefixes of length $n$ of all blocks in $\F$ ending with $0$. Analogously, we define $\F_1$
as the family of prefixes of length $n$ of all blocks in $\F$ ending with $1$, that is,
\begin{align*}
\F_0&=\{w\in\{0,1\}^n:w0\in\F\},\\
\F_1&=\{w\in\{0,1\}^n:w1\in\F\}.
\end{align*}
Since evidently, $|\F|=|\F_0|+|\F_1|$, we may apply the inductive assumption to infer that our lemma holds whenever $|\F_0|=0$ or $|\F_1|=0$.
So assume that $|\F_0|\neq0$ and $|\F_1|\neq0$. Then every set which is independent for $\F_0$ or $\F_1$ is also independent for
$\F$. Unfortunately, there are sets which are independent for both $\F_0$ and $\F_1$ (the empty set, for example).
However, if $J\in\Ind(\F_0)\cap\Ind(\F_1)$, then $J\cup\{n+1\}$ is independent for $\F$, but $J\cup\{n+1\}$ is neither in $\Ind(\F_0)$ nor in $\Ind(\F_1)$.
Therefore, by the inclusion-exclusion formula
\[
|\Ind(\F)|\ge |\Ind(\F_0)\cup\Ind(\F_1)|+|\Ind(\F_0)\cap\Ind(\F_1)|=|\Ind(\F_0)|+|\Ind(\F_1)|.
\]
But applying the inductive hypothesis we obtain
\[
|\Ind(\F_0)|+|\Ind(\F_1)|\ge |\F_0|+|\F_1|=|\F|
\]
and the proof is finished.
\end{proof}

\begin{lemma}\label{indep}
Let $\F\subset\{0,1\}^n$ be a family of binary blocks of length $n\ge 1$.
If for some $1\le k\le n$ we have
\[
|\F|>\sum_{j=0}^{k-1}\binom{n}{j},
\]
then $\F$ is independent over some set of cardinality $k$.
\end{lemma}

\begin{proof}
By the previous lemma and our assumption we have
\[
|\Ind(\F)|\ge |\F|>\sum_{j=0}^{k-1}\binom{n}{j}.
\]
As the last sum is the cardinality of the family of all subsets of $\{1,\ldots,n\}$ with less than $k$ elements there must be a set with $k$ elements in $\Ind(\F)$.
\end{proof}

As a corollary of Sauer-Perles-Shelah Lemma and Lemma \ref{suma} we obtain a variant of the Karpovsky-Milman Lemma \cite{KM}.

\begin{lemma}\label{last}
Let $X$ be a binary shift with positive topological entropy.
Then there is an $\eps>0$ such that for every $n\geq 1$ there is a set $J\subset \{1,\ldots,n\}$ with $\lfloor\eps n\rfloor$ elements which is an independence set for $X$.
\end{lemma}

\begin{proof}
Take $0<\eps\le 1/2$ such that $0<H(\eps)<h(X)$, where $H(\eps)=-\eps\log \eps - (1-\eps)\log (1-\eps)$.
By Lemma \ref{suma} we have
\[
\sum_{j=0}^{\lfloor n \eps \rfloor}\binom{n}{j}\le 2^{n\cdot H(\eps)}<2^{nh(X)}\le |\Bl_n(X)|,
\]
for any $n\geq 1$.
It follows from Lemma \ref{indep} that there is a set of independence for $\Bl_n(X)$ with $\lfloor\eps n\rfloor$ elements.
\end{proof}

\subsection{Proof of Theorem \ref{main}}
\begin{proof}[Proof of Theorem \ref{main}]
Let $X$ be a binary shift.

Assume first that $X$ has positive entropy. Let $\lang_n$ be the collection of characteristic functions of sets of independence for $\Bl_n(X)$. We can treat each element of $\lang_n$ as a binary block; then $\lang=\bigcup_{n=1}^\infty\lang_n$ is a factorial and prolongable binary language. We denote the shift space it defines by $I_X$.

By compactness of the full shift and the fact that a point $x\in \Al^\N$ is in $X_\mathcal{L}$ if and only if $x_{[i,j]}\in\mathcal{L}$ for all $i,j\in\N$ with $ i < j$, elements of $I_X$ may be identified with characteristic functions of independence sets of $X$.

It follows from Lemma \ref{last} that $\adsymbol_1(I_X)>0$. Using Theorem \ref{thm:density} we may now fix an element of $I_X$ which is a characteristic function of $J$, an independence set for $X$, such that $\sh(J)=\ad(J)=\adsymbol_1(I_X)>0$.

If, on the other hand, $J$ is an independence set for $X$ such that $\sh(J)=\ad(J)>\delta>0$, then for every $n\in\N$ there are at least $n\delta$ elements of $J$ in $\{1,\ldots,n\}$, and therefore $|\Bl_n(X)|\geq 2^{n\delta}$ for every $n\in\N$. This implies that $h(X)\geq\delta$ and completes the proof.
\end{proof}

\begin{remark} It follows from the above proof of Theorem \ref{main} that for every binary shift $X$ one has
\[
h(X)\ge \adsymbol_1(I_X).
\]
One may wonder, whether the above inequality is, in fact, an equality.
It holds, for example, for the square-free flow $S$ (see \cite{Peckner}).
Unfortunately, this is not always true. Consider the golden mean shift $G$ (the shift defined by taking $\mathcal{F}=\{11\}$ as the set of forbidden blocks). Note that it is \emph{hereditary}\footnote{Hereditary shifts were introduced in \cite{KerrLi} and examined in \cite{Kwietniak}.}, that is, given a word in its language, one can replace any number of $1$'s by $0$'s and the resulting word will still belong to the language. Therefore $I_G=G$. It is known that $h(G)=\log((1+\sqrt{5})/2)$, but it is easy to see that $\adsymbol_1(I_G)=1/2$.
\end{remark}

\section{Entropy, Topological Mixing, Ergodic Measures, and Chaos}
In this section we construct an example of a shift space
 with the same properties as the symbolic system constructed in \cite{Weiss73} by Weiss.
We begin by defining inductively a sequence of shift spaces $X_k$ so that $X_0\subset X_1\subset\ldots$  and these inclusions are strict.
After completing the induction we define
\[
X=\overline{\bigcup_{n=0}^\infty X_n}.
\]
This means that
$\Bl(X)=\Bl(X_0)\cup \Bl(X_1)\cup \Bl(X_2)\cup\ldots$.
It will be clear that at each step of our inductive construction some words are added to $\Bl(X_n)$ to form $\Bl(X_{n+1})$.
We say that a block $u\in\{0,1\}^*$ with $||u||_1>0$ is \emph{added at step $n\in N$} if
$u\in \Bl(X_{n})\setminus \Bl(X_{n-1})$. Any block $u\in\Bl(X_0)$ with $||u||_1>0$ is \emph{added at step $0$} by definition.
Recall that $||w||_{1}$ denotes the number of $1$'s in a word $w$. A word $w$ is a~\emph{from-$1$-to-$1$} word if $w=1$, or $w=1v1$
for some $v\in\{0,1\}^*$. Let $Y$ be a shift space over $\{0,1\}$. Let $\G_n(Y)$ ($\G^*_n(Y)$) denote the set of all from-$1$-to-$1$ words in $Y$ with at most (exactly) $n$ occurrences of $1$.
The following two properties of blocks added at step $n$ are easy consequences of the definition of $X_n$ presented below. They will be proved 
along the lines of the definition.
\begin{lemma}\label{lem:one}
If a block $u$ is added at step $n\ge 1$, then there is a prefix of $u$ of the form
$0^{\alpha}\bar{u}'0^\beta\bar{u}''$ for some $\bar{u}',\bar{u}''\in\G_{2^{n-1}}(X_{n-1})$, $\alpha\ge 0$ and $\beta\ge 2^{2n-2}$.
It follows that:
\begin{enumerate}
\item \label{added1} $2^{n-1}< ||u||_1$;
\item \label{added2} $\displaystyle\frac{||u||_1}{|u|}\le\frac{2^{n}}{2^{n}+2^{2n-2}}= \frac{1}{1+2^{n-2}}$.
\end{enumerate}
\end{lemma}

\subsection{Construction of $X$} We inductively construct an increasing sequence of shift spaces $X_0\subset X_1\subset\ldots$ and then define $X$ as the closure of the union of all $X_n$'s.

\subsubsection*{Outline of the construction}
Recall that a shift space $X$ is mixing if for every blocks $u,v\in\Bl(X)$ there is an $N\in \N$ such that for each $n\ge N$ one can find a block $w$ of length $n$ such that $uwv\in\Bl(X)$. We call such a word $w$ the \emph{$n$-transition block} from $u$ to $v$. The periodic points are dense in a shift space $X$ is for every block $u\in\Bl(X)$ there is a block $v$ such that $(uv)^\infty\in X$. If $x\in X$ is a periodic point, then the \emph{prime period} of $x$ is the length of the shortest block $w\in \Bl(X)$ satisfying $x=w^\infty$.

In our construction  the $n$-transition block can be always a block of $0$'s. The length $N$ depends on the number of $1$'s which occur in $u$ and $v$, namely we choose $N$ such that the relative density of $1$'s in $u0^Nv$ is small. The periodic points are dense in our shift space but the density of $1$'s in a periodic point decreases
very fast when the prime period of the point grows. We achieve this by adding periodic points of the form $x=(u0^n)^\infty$ with $n$ large enough to force the relative frequency of $1$'s in $x$ to be small.


\subsubsection*{Definition of $X_{0}$}
We define $X_0$ to be the set of all sequences in the full shift over $\{0,1\}$ with at most one appearance of the symbol $1$. That is,
$X_0=\{0^\infty\}\cup\{0^\alpha10^\infty:\alpha\ge 0\}$. %
Note that
\[
\Bl(X_0)=\{0^\alpha:\alpha\ge 0\}\cup \{0^\alpha10^\beta:\alpha,\beta\ge 0\}.
\]
Before we present the details in full generality let us first work out the special case of $X_1$.

\subsubsection*{Definition of $X_{1}$}
We add to $X_0$ the orbits of points of the form:
\[
(10^\beta)^\infty, \quad\text{or}\quad 0^\alpha10^\beta10^\infty, \text{ where }\beta\ge 1,\,\alpha\ge 0.
\]
Note that every block $u$ added at step $1$ has at least two occurrences of the symbol $1$, and fulfills
\[
||u||_1/|u|\le \frac{2}{3}.
\]
The bound above is the best possible, since it is attained by the block $101$ added at step $1$.
Therefore all assertions of Lemma \ref{lem:one} hold.

\subsubsection*{Definition of $X_{n+1}$}
For the inductive step, given $X_n$ and $n\ge 0$, we construct a~shift space $Y_{n+1}$ and set $X_{n+1}=X_n\cup Y_{n+1}$.
To define $Y_{n+1}$ we specify a set of auxiliary points and then let $Y_{n+1}$ be the closure
of the set of auxiliary points. 

There will be two types of auxiliary points. We call them \emph{periodic} and \emph{joining} auxiliary points.
Let $P'_{n+1}$ be the set of all periodic points
of the form $(u0^{k})^\infty$, where $k\ge 2^{2n}$ and $u\in\G_{2^n}(X_n)$.
A point $x$ is a periodic auxiliary point if it belongs to the orbit of some point in $P'_{n+1}$. We denote the set of auxiliary periodic points
by $P_{n+1}$. We have
\[
P_{n+1}=\bigcup_{l=0}^\infty\sigma^l(\{(u0^{k})^\infty:k\ge 2^{2n},\,u\in\G_{2^n}(X_n) \}).
\]
It is clear that $P_{n+1}$ is shift invariant.
The set of joining auxiliary points is given by
$J_{n+1}=J'_{n+1}\cup J''_{n+1}$,
where
\begin{align*}
J'_{n+1}&=\{0^\alpha u0^\infty:\alpha\ge 0,\,u\in\G_{2^n}(X_n)\},\\
J''_{n+1}&=\{0^\alpha u0^{\beta}v0^\infty:\alpha\ge 0,\,\beta\ge 2^{2n},\,u,v\in\G_{2^n}(X_n)\}.
\end{align*}
Points from $J''_{n+1}$ guarantee that the mixing condition holds for pairs of blocks $u, v$ with at most
$2^n$ occurrences of $1$. But $J''_{n+1}$ is not shift invariant and orbits of points from $J''_{n+1}$ end up in $J'_{n+1}\cup\{0^\infty\}$
after a finite number of shift operations.

We claim that
\begin{align}
\overline{P}_{n+1}&\subset P_{n+1}\cup \overline{J}_{n+1}, \label{closures2}\\
\overline{J}_{n+1}&=J_{n+1}\cup\{0^\infty\}. \label{closures1}
\end{align}
For the proof of the inclusion in \eqref{closures2} assume that a~sequence $\{x^{(m)}\}_{m=1}^\infty\subset P_{n+1}$ converges to $x$.
Without loss of generality we can assume that
\[
x^{(m)}=u^{(m)}_{[t(m),\ell(m)]}\left(u^{(m)}0^{\beta_m}\right)^{\infty}
\]
where $u^{(m)}\in\G_{2^n}(X_n)$ is a block of length $\ell(m)$, $1\le t(m)\le ell(m)$ and $\beta_m\geq 2^{2m} $.
We also assume that $u^{(m)}0^{\beta_m}$ is primitive, that is, it is not a concatenation of two or more copies of some block over $\{0,1\}$.

There are two possibilities:

\textbf{Case 1: $\G^*_{2^{n+1}}(x)\neq \emptyset$}. Let $N$ be the smallest integer such that $x_{[1,N]}\in\G^*_{2^{n+1}}(x)$. But $x^{(m)}\to x$ as $m\to\infty$, thus for all sufficiently large $m$ we have $x^{(m)}_{[1,N]}=x_{[1,N]}$. In particular, $x_{[1,N]}\in \Bl \left(x^{(m)}\right)$. Consequently, $u^{(m)}0^{\beta_m}$ is a subblock of $x_{[1,N]}$. There are only finitely many subblock of $x_{[1,N]}$ of the form $u0^k$ where $k\ge 2^m$ and $u\in\G_{2^n}(X_n)$. Therefore there are only finitely many $y\in P_{n+1}$ such that $\G^*_{2^{n+1}}(y)$ contains $x_{[1,N]}$. It follows that $x^{(m)}$ is eventually constant, hence $x\in P_{n+1}$.

\textbf{Case 2: $\G^*_{2^{n+1}}(x)= \emptyset$}. It is easy to see that in this case $x\in \overline{J}_{n+1}$.

The proof of \eqref{closures1} follows the same lines as the proof above. We leave the details to the reader.

We set \[Y_{n+1}=\overline{P_{n+1}\cup J_{n+1}}\;\;\text{and} \;\;X_{n+1}=X_n\cup Y_{n+1}.\] Clearly, $X_n\subset X_{n+1}$.  This completes the induction step. Observe that any block $u$ added at step $n+1$ has a prefix
$0^{\alpha}\bar{u}'0^\beta\bar{u}''$ for some $\bar{u}',\bar{u}''\in\G_{2^{n}}(X_{n})$, $\alpha\ge 0$, $\beta\ge 2^{2n}$. Furthermore, $u$ fulfills $||u||_1>2^n$ as otherwise $u$ would be added at earlier step. Therefore
\[
\frac{||u||_1}{|u|}\le \frac{||u'||_1+||u''||_1}{|u'|+|u''|+{\beta}}.
\]
But $||u'||_1$ and $||u''||_1$ are bounded by $2^n$, and we conclude that
\[
\frac{||u||_1}{|u|}\le \frac{2^{n+1}}{2^{n+1}+2^{2n}}=\frac{1}{1+2^{n-1}},
\]
as claimed in Lemma \ref{lem:one}.
We note the following consequence of the formulas  \eqref{closures1} and \eqref{closures2}.
\begin{corollary}
 If $x\in X_n$ and the set $\{j\in\N:x_j=1\}$  is infinite, then $x$ must be a periodic point.
\end{corollary}

\subsection{Properties of $X$}
We describe the properties of $X$ as a sequence of lemmas.
\begin{lemma}
Periodic points are dense in $X$.
\end{lemma}
\begin{proof}
Take any $u\in\Bl(X)$. Assume that $u$ is added at step $n$ and let $k=||u||_1$.
If $n=k=0$, then $0^\infty$ is a periodic point in the cylinder of $u$. If $k>0$, then we write $u=0^\alpha \bar{u} 0^\beta$ for some $\bar{u}\in\G_k(X)$, $\alpha,\beta\ge 0$. We set $m=\lceil\log_2 k\rceil$ and note that $m\ge n$. Hence
$\bar{u}\in \G_k(X_n)\subset \G_{2^m}(X_m)$. We may infer that
\[
0^\alpha \left(\bar{u} 0^{\alpha+\beta+2^{2m}}\right)^\infty = (u0^{2^{2m}})^\infty\in P_{m+1},
\]
and therefore there is a periodic point of $X$ in the cylinder of $u$.
\end{proof}
\begin{lemma}
The shift space $X$ is topologically mixing.
\end{lemma}
\begin{proof}
Let $u$ and $v$ be two nonempty blocks in $\Bl(X)$. Assume that $||u||_1>0$ and $||v||_1>0$. Consequently,
there is a~$k\ge 0$ such that $u=0^a\bar{u} 0^b$ and $v=0^c\bar{v}0^d$ for some $a,b,c,d\ge 0$ and $\bar{u},\bar{v}\in\G_{2^k}(X_k)$. Therefore $u0^{\beta}v0^\infty\in J_{n+1}\subset X$ for all $\beta\ge 2^{2k}$. If for some $\alpha\in\N$ we have $u=0^\alpha$ or $v=0^\alpha$, then our reasoning is similar. Therefore $X$ is topologically mixing.
\end{proof}

\begin{lemma}\label{lem:periodic-or-zero1}
If $x\in X$, then either $x$ is periodic or
the asymptotic density of the set $\chi_1(x)=\{j\in\N: x_j=1\}$ equals zero.
\end{lemma}
\begin{proof}
Assume that the set $\chi_1(x)$ is infinite.
Denote by $\lambda_n$ the step at which $x_{[1,n]}$ was added. The sequence $\lambda_n$ is nondecreasing, hence it is either bounded, or $\lambda_n\nearrow \infty$. In the former case $x\in X_m$ for some $m\in\N$, hence $x$ must be periodic. In the later case we have
\[
\frac{\left|\{1\le j\le n: x_j=1\}\right|}{n}\le \frac{1}{1+2^{\lambda_n-2}},
\]
which monotonically tends to zero as $n$ goes to infinity by Lemma \ref{lem:one}.
\end{proof}

\begin{lemma}\label{lem:all-generic-only-periodic-ergodic}
Every point $x\in X$ is a generic point for some $\sigma$-invariant ergodic measure $\mu_x$ on $X$
and all ergodic invariant measures of $(X,\sigma)$ are supported on single periodic points.
\end{lemma}
\begin{proof}If $x$ is a periodic point, then it is clear that $x$ is generic.
If $x$ is not a~periodic point, then for each $k\in\N$ the asymptotic density of the set $\chi^k_1(x)=\{j\in\N: x_{[j,j+k)}=0^k\}$ equals one. It follows
that $x$ is generic for the Dirac measure concentrated on the fixed point $0^\infty$. It follows that there are no other ergodic invariant measures.
\end{proof}


\begin{lemma}\label{lem:noDC3}
The shift space $X$ has no DC$3$ pair.
\end{lemma}
\begin{proof}
Fix $k\in N$ and take any $x,y\in X$.
We define
\[
\Equal^k(x,y)=\{j\in\mathbb{N}:x_{[j,j+k)}= y_{[j,j+k)}\}.
\]
It is easy to see that in the case of symbolic dynamical systems equipped in the metric $\rho$ defined by
\eqref{eq:rho} for each $k\in \N$ and $t\in (2^{-k},2^{-k+1}]$ we have
\begin{align}\label{eq:symbolic_distribution}
F_{xy}(t)&=\ld(\Equal^k(x,y))&\text{and}&& F^*_{xy}(t)&=\ud(\Equal^k(x,y)).
\end{align}
We will show that $\Equal^k(x,y)$ has a well defined asymptotic density.
Let $u_j=x_{[j,j+k)}$ and $v_j=y_{[j,j+k)}$ for $j\in\N$.
If $x$ and $y$ are periodic, then obviously $\Equal^k(x,y)$ has an asymptotic density for each $k$.
For the remaining case we assume, without loss of generality, that $x$ is generic for the Dirac mass at $0^\infty$.
Observe that
\begin{align*}
\Equal^k(x,y)&=\{j\in\mathbb{N}:u_j=v_j\}\\
&=\{j\in\mathbb{N}:u_j=0^k=v_j\}\cup \{j\in\mathbb{N}:u_j= v_j\neq 0^k\}.
\end{align*}
The set $\{j\in\mathbb{N}:u_j= v_j\neq 0^k\}$ is contained in $\{j\in\mathbb{N}:u_j\neq 0^k\}$, hence has asymptotic density zero. We also know that $d(\{j\in\mathbb{N}:u_j= 0^k\})=1$ and $d(\{j\in\mathbb{N}:v_j=0^k\})$ exists, therefore $\Equal^k(x,y)$ has a well defined asymptotic density.
Since $k$ was arbitrary, we apply \eqref{eq:symbolic_distribution} to conclude that $F_{xy}(t)=F^*_{xy}(t)$ for each $t\in\R$.
\end{proof}

The following result may be proved directly, or using a result of Downarowicz \cite{Downarowicz}, who showed that any positive entropy system has many DC$2$-pairs.
\begin{lemma}\label{cor:zero-entropy}
The shift space $X$ has topological entropy zero.
\end{lemma}
\begin{proof}
Note that all ergodic measures of $X$ are concentrated on periodic points, hence their measure-theoretic entropy is equal zero.
By the variational principle the topological entropy is also equal zero.
\end{proof}

Finally, we note the following corollary of Lemma~\ref{lem:all-generic-only-periodic-ergodic}.

\begin{corollary}
There is no ergodic invariant measure with support $X$.
\end{corollary}

\begin{lemma}
The shift space $X$ is not $\omega^*$-chaotic.
\end{lemma}
\begin{proof}
The only minimal sets in $X$ are periodic orbits. Therefore no $\omega$-limit set contains an infinite minimal set.
\end{proof}

The following lemma is well-known (see \cite[Chap. 4, Problem 41]{KT}).
\begin{lemma}\label{lem:independent-family}
There exists an uncountable set $\Gamma$ and a family $\{B(\gamma)\}_{\gamma\in\Gamma}$ of infinite subsets of $\mathbb{N}$ such that for any $\alpha,\beta\in \Gamma$ with $\alpha\neq \beta$ we have
\begin{enumerate}
\item  $B(\alpha)\cap B(\beta)$ is infinite,
\item  $B(\alpha)\setminus B(\beta)$ is infinite.
\end{enumerate}
\end{lemma}

We now proceed to our last result.

\begin{lemma}
The shift space $X$ is $\omega$-chaotic.
\end{lemma}
\begin{proof}
We first define for any infinite set $S\subset\N$ an infinite transitive subsystem $X_S\subset X$. 
To do this we proceed as with the construction of $X$. We start with $X^S_0=X_0$.
During the induction step, we define the set of joining auxiliary points  by
$J^S_{n+1}=\dot{J}^S_{n+1}\cup \ddot{J}^S_{n+1}$, where
\begin{align*}
\dot{J}^S_{n+1}&=\{0^\alpha u0^\infty:\alpha\ge 0,\,u\in\G_{2^n}(X^S_n)\},\\
\ddot{J}^S_{n+1}&=\{0^\alpha u0^{\beta}v0^\infty:\alpha\ge 0,\,\beta\ge 2^{2n},\,\beta \in S,\,u,v\in\G_{2^n}(X^S_n)\}.
\end{align*}
Note that $J^S_{n+1}\subset J_{n+1}$ and $J^S_{n+1}$ nonempty because $S$ is infinite. The main difference from the previous construction is the condition $\beta\in S$.
We do not add any periodic points. As above we set $X_{n+1}^S=X_n^S\cup \overline{J^S_{n+1}}$. Then
\[
X_S=\overline{\bigcup_{n=0}^\infty X^S_n}.
\]

We inductively define the family of \emph{allowed} blocks. We say that every block $w\in\{0,1\}^*$ with $||w||_1\le 1$ is allowed. 
Assume that for some $n\in\N$ we have defined which blocks $w$ over $\{0,1\}$ with at most $2^{n-1}$ ones are allowed. We say that a block $w\in\{0,1\}^*$ with $2^{n-1}<||w||_1\le 2^n$ is allowed if there are allowed from-$1$-to-$1$ blocks $u_1$ and $u_2$ with $1\le ||u_j||_1\le 2^{n-1}$ for $j=1,2$, $a,b\ge 0$, and $\beta\in S$ such that $\beta\ge 2^{2n}$ and $w=0^au_10^\beta u_20^b$.
A straightforward (but tedious) analysis of the definition of $X_S$ shows that block $w$ belongs to the language $\Bl(X_S)$ of $X_S$ if and only if it is an allowed block. This finishes the construction of $X_S$. 

We claim that $X_S$ is transitive. To this and take two allowed blocks $u$ and $v$. Without loss of generality we may assume that $u$ and $v$ are from-$1$-to-$1$ blocks. Let $i=||u||_1$ and $j=||v||_1$. Pick $N$ such that
$\max\{i,j\}\le 2^N$. Let $\beta\in S$ be such that $\beta\ge 2^{2N}$. Then $u0^\beta v$ is an allowed word, which proves the transitivity of $X_S$.
Again, we used the fact that $S$ is infinite. It follows that $X_S$ is a topologically transitive and infinite shift space. Therefore $X_S$ is uncountable.

Note that if $S$ and $T$ are infinite subsets of $\N$ such that $S\cap T$ is also infinite, then $X_{S\cap T}\subset X_S\cap X_T$ (here one can write an equality, but this inclusion suffices for our purposes). It is also clear that if $S\subset T\subset \N$ is infinite, then $X_S\subset X_T$. Furthermore, if $S,T\subset \N$ are infinite and disjoint, then $X_S\cap X_{T}=X_0$.
To see this, observe that $X_0$ is contained in $X_S$ for every infinite $S$. On the other hand, if $x\in (X_S\cap X_{T})\setminus X_0$, then we can find $w=10^\beta 1\in \Bl(X_S)\cap\Bl(X_T)$, which is impossible since then $\beta\in S\cap T=\emptyset$. The same argument shows that if $S\neq T$, then $X_S\neq X_T$.

Recall that $X_0$ is a countable shift space.
In particular, if $S,T\subset \N$ are infinite and disjoint, then $X_S\setminus X_{T}=X_S\setminus X_0$ is uncountable.

Let $\{B(\gamma)\}_{\gamma\in\Gamma}$ be a family provided by Lemma \ref{lem:independent-family}. For any $\alpha,\beta\in \Gamma$ with $\alpha\neq \beta$ we define $C(\alpha,\beta)=B(\alpha)\cap B(\beta)$ and $D(\alpha,\beta)=B(\alpha)\setminus B(\beta)$.

Let $X_\gamma$ be the shift space $X_S$ constructed for $S=B(\gamma)$, where $\gamma\in \Gamma$.
Using transitivity, for any $\gamma\in\Gamma$ we pick a point $x_\gamma\in X_\gamma$ such that its omega-limit set, $\omega(x_\gamma)=X_{\gamma}$.
We claim that $\{x_\gamma\}_{\gamma\in\Gamma}$ is an $\omega$-chaotic set. Clearly, $\{x_\gamma\}_{\gamma\in\Gamma}$ is uncountable, because $X_{\gamma}\neq X_{\gamma'}$ if $\gamma\neq\gamma'$. Furthermore, $B(\gamma)$ is infinite, therefore $\omega(x_\gamma)$ is uncountable for each $\gamma\in\Gamma$ and is not contained in the set of periodic points. Now note that for every $\alpha,\beta\in\Gamma$ with $\alpha\neq\beta$ we have
$\omega(x_\alpha)\cap \omega(x_\beta)\supset X_{C(\alpha,\beta)}$, which is uncountable because $C(\alpha,\beta)$ is infinite and $\omega(x_\alpha)\setminus \omega(x_\beta)$ is also uncountable (it contains all transitive points $X_{D(\alpha,\beta)}$, because $D(\alpha,\beta)$ is infinite and disjoint with $B(\beta)$).
Hence $\{x_\gamma:\gamma\in\Gamma\}$ is an $\omega$-chaotic set, which proves that $X$ is $\omega$-chaotic.
\end{proof}

\section*
{Acknowledgment.} We would like to thank the referee for the thorough, constructive and helpful comments and suggestions on the manuscript. We are grateful to: Jakub Byszewski, Vaughn Climenhaga, Jakub Konieczny, Marcin Lara, Martha {\L}{\c{a}}cka, Dariusz Matlak, Ville Sallo and Maciej Ulas for their remarks and suggestions on the present and previous version of this paper. We wish to express our thanks to Benjamin Weiss for several helpful comments concerning his work and our construction in Section 6.
The research of Dominik Kwietniak was supported by the  National Science Centre (NCN) under grant no. DEC-2012/07/E/ST1/00185.
The research of Jian Li  was supported by NSF of Guangdong province (S2013040014084) and Scientific Research Fund of Shantou University (YR13001).


\end{document}